\documentclass[a4paper,10pt,leqno]{article}
\usepackage{amsmath,amsfonts,amssymb,amsthm}
\usepackage{graphicx,subfigure,booktabs,multirow}
\usepackage{color}
\usepackage[textwidth=14.2cm]{geometry}

\newtheorem{theorem}{Theorem}[section]
\newtheorem{lemma}[theorem]{Lemma}
\newtheorem{proposition}[theorem]{Proposition}
\newtheorem{corollary}[theorem]{Corollary}
\newtheorem{remark}[theorem]{Remark}
\newtheorem{definition}{Definition}[section]
\newtheorem{example}{Example}[section]
\newcommand{\R}[0]{\mathbb{R}}
\newcommand{\ten}[1]{\mathcal{#1}}

\title{P-Tensors, P$_0$-Tensors, and \\ Tensor Complementarity Problem\thanks{This research was supported by the National Natural Science Foundation of China (11271084,11301022,11431002), the State Key Laboratory of Rail Traffic Control and Safety, Beijing Jiaotong University (RCS2014ZT20,RCS2014ZZ01), and the Hong Kong Research Grant Council (Grant No.
PolyU 502111, 501212, 501913 and 15302114).}}
\author{
Weiyang Ding\thanks{School of Mathematical Sciences, Fudan University, Shanghai 200433, P. R. China; ({\tt weiyang.ding@gmail.com}).}
\and
Ziyan Luo\thanks{State Key Laboratory of Rail Traffic Control and Safety, Beijing Jiaotong University, Beijing 100044, P.R. China; ({\tt starkeynature@hotmail.com}).}
\and
Liqun Qi\thanks{Department of Applied Mathematics, The Hong Kong Polytechnic University, Hung Hom, Kowloon, Hong Kong. ({\tt liqun.qi@polyu.edu.hk}).}
}
\date{\today}

\begin{document}

\maketitle
\begin{abstract}
The concepts of P- and P$_0$-matrices are generalized to P- and P$_0$-tensors of even and odd orders via homogeneous formulae. Analog to the matrix case, our P-tensor definition encompasses many important classes of tensors such as the positive definite tensors, the nonsingular M-tensors, the nonsingular H-tensors with positive diagonal entries, the strictly diagonally dominant tensors with positive diagonal entries, etc. As even-order symmetric PSD tensors are exactly even-order symmetric P$_0$-tensors, our definition of P$_0$-tensors, to some extent, can be regarded as an extension of PSD tensors for the odd-order case. Along with the basic properties of P- and P$_0$-tensors, the relationship among P$_0$-tensors and other extensions of PSD tensors are then discussed for comparison. Many structured tensors are also shown to be P- and P$_0$-tensors. As a theoretical application, the P-tensor complementarity problem is discussed and shown to possess a nonempty and compact solution set.

  \vskip 12pt
  \noindent {\bf Key words.} { P-tensor, P$_0$-tensor, positive semidefinite tensor, tensor complementarity problem}

  \vskip 12pt
  \noindent {\bf AMS subject classifications. }{15A18, 15A69, 15B48, 90C33}
%15A69:LINEAR AND MULTILINEAR ALGEBRA; MATRIX THEORY- Multilinear algebra, tensor products
%53A45:Classical differential geometry-Vector and tensor analysis

\end{abstract}

%===============================================================================================================================
%===============================================================================================================================

\section{Introduction}

The P-matrix, first introduced by Fiedler and Pt{\'a}k \cite{Fiedler62}, is an important class of special matrices whose determinants of all principal submatrices are positive. As an essential concept in matrix theory (see, e.g., \cite{HuangTZ07,Pena09,Pena11,Rump03,Song99}), P-matrices contain many notable matrices as their subclasses, such as positive definite matrices, nonsingular M-matrices, nonsingular completely positive matrices, strictly diagonally dominant matrices, etc., Besides its significance in matrix analysis, P-matrix also plays an important role in both the theory and the computations of linear complementarity problems (see, e.g., \cite{Berman94,Cottle92,Pang03a,Pang03b}). Given a matrix $A$ and a vector ${\bf q}$, a linear complementarity problem ${\rm LCP}({\bf q},A)$ is to find a vector ${\bf x}$ such that
\begin{equation}\label{eq_lcp}
  {\bf x} \geq {\bf 0},\ A{\bf x} + {\bf q} \geq {\bf 0},\ \langle {\bf x}, A{\bf x} + {\bf q} \rangle = 0,
\end{equation}
where $\langle\cdot,\cdot\rangle$ denotes the inner product of two vectors. The matrix $A$ is a P-matrix if and only if the corresponding ${\rm LCP}({\bf q},A)$ has a unique solution for any vector ${\bf q}$. The P-matrix is also a widely-applied tool in nonlinear complementarity problems. For instance, Qi et al. \cite{Qi00} employed some properties of P-matrices to investigate the convergence of smoothing Newton methods for nonlinear complementarity problems.

With an emerging interest in the assets of multilinear algebra concentrated on the higher-order tensors, more and more structured matrices have been extended to higher-order structured tensors. Very recently, Song and Qi \cite{Song15a} extended the concepts of P- and P$_0$-matrices to P- and P$_0$-tensors. Their applications in tensor complementarity problems (TCP) (an extension of linear complementarity problems as studied in \cite{Che15,Luo15,Song15c,Song15b}) were discussed in \cite{Song15c,Song15b}. As pointed out by Yuan and You \cite{Yuan14}, there are no odd-order P-tensors according to the definition of Song and Qi \cite{Song15a}. This deficiency is more or less due to the involved non-homogeneous term. In this paper, we will introduce a homogenous definition of P- and P$_0$-tensors, and discuss their applications in TCP and others.

Analog to the matrix case, our P-tensor definition encompasses such notable classes as the positive definite tensors, the nonsingular M-tensors, the nonsingular H-tensors with positive diagonal entries, the strictly diagonally dominant tensors with positive diagonal entries, etc. Of particular concern is that our definition of P$_0$-tensors also extend the class of positive semi-definite (PSD) tensors to the odd-order case. It is well-known that PSD tensors have massive applications in optimal control, magnetic resonance imaging, spectral hypergraph theory and tensor complementarity problem, etc. However, the definition of PSD tensors restricts the order of the involved tensors to be even. As even-order symmetric PSD tensors are exactly even-order symmetric P$_0$-tensors, our definition of P$_0$-tensors, to some extent, can be regarded as a counterpart of PSD tensors for the odd-order case.

The rest of the paper is organized as follows. In Section 2, the definitions of P- and P$_0$-tensors are introduced and some basic properties are characterized. In Section 3, relationship among P$_0$-(P-)tensors,  PSD (PD) tensors, and other PSD-related tensors such as (strictly) copositive tensors, (strongly) completely positive tensors, and doubly nonnegative tensors are discussed. In Section 4, some structured P- and P$_0$-tensors are established. Properties on P- and P$_0$-tensor Cones, together with their relationship to other special tensor cones are analyzed in Section 5. Applications of P- and P$_0$-tensors in TCP and others are described in Section 6. Concluding remarks are drawn in Section 7.

%==================================================================================================================================================================================
%==================================================================================================================================================================================

\section{Definitions and Basic Properties}

The concept of P-matrix \cite{Fiedler62} is originally defined by that all of its principal minors are positive. As an important equivalent definition of the P-matrix, we know that a matrix $A$ is a P-matrix if and only if for each nonzero vector ${\bf x} \in \mathbb{R}^n$ there exists some index $i$ such that $x_i (A{\bf x})_i > 0$ (see \cite{Berman94,Fiedler62}). Following this equivalent definition, Song and Qi \cite{Song15a} generalized the concept of P-matrices to higher order tensors and they call an $m^{\rm th}$-order $n$-dimensional tensor $\ten{A}$ a P-tensor if and only if for each nonzero vector ${\bf x} \in \mathbb{R}^n$ there exists some index $i$ such that $x_i (\ten{A}{\bf x}^{m-1})_i > 0$. Nevertheless, it has been shown that there are no odd-order P-tensors according under the Song-Qi's definition (see \cite{Yuan14}). To make up this deficiency, we now propose a modified definition of P-tensors as follows.

\begin{definition}[P- and P$_0$-tensors]\label{def_pten}
  An $m^{\rm th}$-order $n$-dimensional tensor $\ten{A}$ is called a P-tensor, if for each nonzero ${\bf x} \in \mathbb{R}^n$ there exists some index $i$ such that
  \begin{equation}\label{eq_pten}
  x_i^{m-1} (\ten{A}{\bf x}^{m-1})_i > 0.
  \end{equation}
  An $m^{\rm th}$-order $n$-dimensional tensor $\ten{A}$ is called a P$_0$-tensor, if for each nonzero ${\bf x} \in \mathbb{R}^n$ there exists some index $i$ such that
  \begin{equation}\label{eq_p0ten}
  x_i \neq 0 \quad\text{and}\quad x_i^{m-1} (\ten{A}{\bf x}^{m-1})_i \geq 0.
  \end{equation}
\end{definition}

It is apparent to see that our definition is the same as the one of Song and Qi in the even-order case. And the expressions \eqref{eq_pten} and \eqref{eq_p0ten} can be rewritten as
$$
x_i \neq 0 \quad\text{and}\quad (\ten{A}{\bf x}^{m-1})_i > 0\ (\text{or } \geq 0),
$$
respectively. Thus, the properties about the P- and P$_0$-tensors in \cite{Song15a,Song15c,Song15b,Yuan14} are certainly valid for our P- and P$_0$-tensors of even orders. Moreover, we will shortly show that most of those properties also holds for our P- and P$_0$-tensors of odd orders.

%\begin{proposition}\label{prop_nonpten}
%  Let $m$ be an odd number. An $m^{\rm th}$-order $n$-dimensional tensor $\ten{A}$ is not a P$_0$-tensor if and only if there exists an ${\bf y} \in \mathbb{R}^n$ such that for all $i\in [n]$, either $y_i = 0$ or $(\ten{A} {\bf y}^{m-1})_i < 0$.
%\end{proposition}

We start with the spectral property of P- and P$_0$-tensors. According to Qi's work in \cite{Qi05}, for any given $m^{\rm th}$-order $n$-dimensional tensor $\ten{A}$, if there is a real scalar $\lambda \in \mathbb{R}$ and a real vector ${\bf x} \in \mathbb{R}^n$ satisfying
$$
\ten{A}{\bf x}^{m-1} = \lambda {\bf x}^{[m-1]},
$$
then $\lambda$ is called an H-eigenvalue of $\ten{A}$ and ${\bf x}$ is called a corresponding H-eigenvector. Here ${\bf x}^{[m-1]}:=\left(x_1^{m-1},\ldots,x_n^{m-1}\right)^T$. The following theorem shows the positivity (nonnegativity) of the H-eigenvalues of a P-tensor (P$_0$-tensor).

\begin{theorem}\label{thm_heig}
  A P-tensor (P$_0$-tensor) has no non-positive (negative) H-eigenvalues.
\end{theorem}
\begin{proof}
  Let $\ten{A}$ be an $m^{\rm th}$-order $n$-dimensional P-tensor. If $\ten{A}$ has no H-eigenvalue, then the result is trivial. Otherwise, we use $\lambda \in \mathbb{R}$ and ${\bf x} \in \mathbb{R}^n$ to denote an H-eigenvalue of $\ten{A}$ and a corresponding H-eigenvector, respectively. That is, $\ten{A}{\bf x}^{m-1} = \lambda {\bf x}^{[m-1]}$, which implies $x_i^{m-1} (\ten{A}{\bf x}^{m-1})_i = \lambda x_i^{2(m-1)}$ for all $i=1,2,\dots,n$. By definition, we know that ${\bf x}\neq {\bf 0}$. Thus, we can always find some index $i$ such that the left-hand side is positive, which further implies that the H-eigenvalue $\lambda$ must be positive as well.

  The proof of the results about P$_0$-tensors is similar, thus we omit it.
\end{proof}

Moreover, it can be easily verified that the property of a P- or P$_0$-tensor is also preserved by all of its principal subtensors.

\begin{theorem}\label{thm_subten}
  If an $m^{\rm th}$-order $n$-dimensional tensor $\ten{A}$ is a P- or P$_0$-tensor, then each of its principal subtensor is also a P- or P$_0$-tensor, respectively.
\end{theorem}
\begin{proof}
  Let $\ten{A}$ be a P-tensor. Without loss of generality, we can assume that its first $k \times k \times \dots \times k$ subtensor $\ten{A}_{1:k}$ is not a P-tensor, i.e., there exists some nonzero ${\bf y} \in \mathbb{R}^k$ such that $y_i^{m-1} (\ten{A}_{1:k}{\bf y}^{m-1})_i \leq 0$ for all $i=1,2,\dots,n$. Denote ${\bf x} = [{\bf y}^\top,{\bf 0}^\top]^\top \in \mathbb{R}^n$. Then we have
  $$
  \ten{A}{\bf x}^{m-1} = \begin{bmatrix}
    \ten{A}_{1:k} {\bf y}^{m-1} \\ \ast
  \end{bmatrix}.
  $$
  Hence, the products of the entries of ${\bf x}^{[m-1]}$ and the corresponding entries of $\ten{A}{\bf x}^{m-1}$ are
  $$
  x_i^{m-1} (\ten{A}{\bf x}^{m-1})_i = \left\{\begin{array}{ll}
    y_i^{m-1} (\ten{A}_{1:k}{\bf y}^{m-1})_i \leq 0, & i = 1,2,\dots,k, \\
    0, & i = k+1,k+2,\dots,n.
  \end{array}\right.
  $$
  This contradicts that $\ten{A}$ is a P-tensor. Therefore, each principal subtensor of a P-tensor is also a P-tensor.
\end{proof}

From this theorem, we can directly point out that all the diagonal entries of a P-tensor (or P$_0$-tensor) are positive (or nonnegative), since each diagonal entry is a smallest principal subtensor. This is one of the most easily checkable necessary conditions of P- and P$_0$-tensors.

The following equivalent definition of P-tensors describes the positive definiteness of a P-tensor corresponding to a fixed vector.
\begin{theorem}
  An $m^{\rm th}$-order $n$-dimensional tensor $\ten{A}$ is a P-tensor if and only if for each nonzero vector ${\bf x}\in \mathbb{R}^n$, there exists a positive diagonal matrix $D_{\bf x}$ such that
  $$
  \big\langle D_{\bf x} {\bf x}^{[m-1]}, \ten{A}{\bf x}^{m-1} \big\rangle > 0.
  $$
\end{theorem}
\begin{proof}
  For a tensor $\ten{A}$ and a vector ${\bf x}$, denote
  $$
  \delta_i = \left\{\begin{array}{ll}
    1, & x_i^{m-1} (\ten{A}{\bf x}^{m-1})_i > 0, \\
    0, & x_i^{m-1} (\ten{A}{\bf x}^{m-1})_i \leq 0.
  \end{array}\right.
  $$
  By the definition, if $\ten{A}$ is a P-tensor, then there must be some index $i$ such that $\delta_i > 0$. Thus it holds that $\sum_{i=1}^n \delta_i x_i^{m-1} (\ten{A}{\bf x}^{m-1})_i > 0$. Let
  $$
  0 < \epsilon < \frac{\sum_{i=1}^n \delta_i x_i^{m-1} (\ten{A}{\bf x}^{m-1})_i}{\big|\sum_{i=1}^n x_i^{m-1} (\ten{A}{\bf x}^{m-1})_i\big|}.
  $$
  Hence, if we denote $D_{\bf x} = {\rm diag}(\delta_1+\epsilon,\delta_2+\epsilon,\dots,\delta_n+\epsilon)$, then the diagonal entries of $D_{\bf x}$ are positive and
  $$
  \big\langle D_{\bf x} {\bf x}^{[m-1]}, \ten{A}{\bf x}^{m-1} \big\rangle = \sum_{i=1}^n (\delta_i+\epsilon) x_i^{m-1} (\ten{A}{\bf x}^{m-1})_i > 0.
  $$
  The reverse part is obvious, thus we omit its proof.
\end{proof}

%===============================================================================================================================
%===============================================================================================================================

\section{Comparison with Extensions of PSD Tensors}

Positive semidefinite (PSD) tensors were first introduced and systemically investigated by Qi \cite{Qi05}, which have wide applications in optimal control (see \cite{Anderson75}), magnetic resonance imaging \cite{QYW10}, spectral hypergraph theory \cite{Hu15,Qi2014,Shao15}, tensor complementarity problem \cite{Che15,Luo15,Song15b,Song15c}, etc.  PSD tensors are a generalization of PSD matrices and are defined as follows.

\begin{definition}[Positive definite \cite{Qi05,CQi2014}]\label{def_psd}
  An $m^{\rm th}$-order $n$-dimensional tensor $\ten{A}$ is called positive definite (PD) if $\ten{A}{\bf x}^{m} > 0$ for any nonzero ${\bf x}\in \mathbb{R}^n$; And it is called positive semidefinite (PSD) if $\ten{A}{\bf x}^{m} \geq 0$ for any ${\bf x}\in \mathbb{R}^n$.
\end{definition}

Actually, we can see from the definition that a PSD tensor is not necessarily symmetric. However, we are still most interested in the symmetric case, otherwise we can study the symmetrization (see \cite{CQi2014} for details) of a nonsymmetric PSD tensor instead. Even without the symmetry, it is easily concluded that a PD or PSD tensor is also a P- or P$_0$-tensor, respectively.

\begin{proposition}\label{prop_psd}
  A positive definite tensor is a P-tensor, and a positive semidefinite tensor is a P$_0$-tensor.
\end{proposition}
\begin{proof}
  If $\ten{A}$ is a zero tensor, then this is true obviously. Let $\ten{A}$ be an $m^{\rm th}$-order $n$-dimensional nonzero positive semidefinite tensor. Thus $m$ is even. Assume on the contrary that $\ten{A}$ is not a P$_0$-tensor. Then we can find some nonzero ${\bf x}\in \mathbb{R}^n$ such that
  $$
  x_i^{m-1} (\ten{A}{\bf x}^{m-1})_i<0,\quad \forall i\in {\rm supp}({\bf x}):=\{i \in [n]: x_i\neq 0\}.
  $$
  Since $m$ is even, the above equation is equivalent to
  $$
  x_i  (\ten{A}{\bf x}^{m-1})_i<0,\quad \forall i\in {\rm supp}({\bf x}),
  $$
  which further implies that $\ten{A}x^m = \sum_{i=1}^n x_i  (\ten{A}{\bf x}^{m-1})_i < 0$. This contradicts to the positive semidefiniteness of $\ten{A}$. Thus $\ten{A}$ is a P$_0$-tensor. Similarly, we can prove the PD case.
\end{proof}

The indeed incompleteness in the definition of PSD tensors is the absence of odd-order PSD tensors, which can be observed apparently by replacing ${\bf x}$ by $-{\bf x}$ when the order $m$ is odd. Nevertheless, the odd-order cases are not allowed to be ignored. For instance, there is no intrinsical gap between the definitions and properties of the even-order uniform hypergraphs and the odd-order ones. And it is well-known that the Laplacian tensor of an even-order uniform hypergraph is PSD, but we have not got a parallel result for the Laplacian tensor of an odd-order uniform hypergraph yet.

Song and Qi \cite{Song15a} pointed out that an even-order symmetric tensor is PD or PSD if and only if it is a P- or P$_0$-tensor, respectively, since a P- or P$_0$-tensor has no nonpositive or negative H-eigenvalues, respectively. Thereby, the equivalence between the PD tensors and the even-order symmetric P-tensors inspires that the odd-order symmetric P-tensors may be a good extension of the positive definiteness to the odd-order case. Apart from symmetric P- and P$_0$-tensors, we have so far other three kinds of extensions of the PD and PSD tensor to the odd-order case, i.e., (strictly) copositive tensors, (strongly) completely positive tensors, and (strongly) doubly nonnegative tensors. We shall compare our extension with those approaches in this section.

Copositivity might be the first attempt to generalize the concept of positive definiteness (see \cite{Qi13,Song15d}). Studying on copositive tensors is restricted on the nonnegative orthant $\mathbb{R}_+^n := \{{\bf x} \in \mathbb{R}^n:\, x_i \geq 0,\,i=1,2,\dots,n\}$, which thus avoids the ill-posedness of the PSD tensors of odd orders and admits the PSD tensor cone as a subset.

\begin{definition}[Copositive \cite{Qi13}]
  An $m^{\rm th}$-order $n$-dimensional tensor $\ten{A}$ is called copositive if $\ten{A}{\bf x}^{m} \geq 0$ for any ${\bf x}\in \mathbb{R}_+^n$; And it is called strictly copositive if $\ten{A}{\bf x}^{m} > 0$ for any nonzero ${\bf x}\in \mathbb{R}_+^n$.
\end{definition}

For even-order tensors, the definitions directly imply that a PSD tensor is copositive. Moreover, we shall see shortly that this is also true for odd-order P-tensors. We need a lemma about the spectra of a copositive tensor from Song and Qi \cite[Theorems 4.1 and 4.2]{Song15d}.

\begin{lemma}[\cite{Song15d}]
  Let $\ten{A}$ be an $m^{\rm th}$-order $n$-dimensional symmetric tensor. Then $\ten{A}$ is copositive if and only if every principal subtensor has no negative H-eigenvalue with a positive H-eigenvector; And $\ten{A}$ is strictly copositive if and only if every principal subtensor has no nonpositive H-eigenvalue with a positive H-eigenvector.
\end{lemma}

By Theorem \ref{thm_subten}, each principal subtensor of a P- or P$_0$-tensor is also a P- or P$_0$-tensor, thus has no nonpositive or negative H-eigenvalue, respectively. Then the relationship between P- or P$_0$-tensors and (strictly) copositive tensors follows directly by the above lemma.

\begin{proposition}\label{prop_cop}
  A symmetric P-tensor is strictly copositive, and a symmetric P$_0$-tensor is copositive.
\end{proposition}

Completely positive tensors, introduced by Qi et al. \cite{Xu14}, are defined by sums of nonnegative rank-one tensors. For several vectors ${\bf x}_k \in \mathbb{R}^{n_k}$ ($k=1,2,\dots,m$), their outer product ${\bf x}_1 \otimes {\bf x}_2 \otimes \dots \otimes {\bf x}_m$ is an $m^{\rm th}$-order rank-one tensor with
$$
({\bf x}_1 \otimes {\bf x}_2 \otimes \dots \otimes {\bf x}_m)_{i_1 i_2 \dots i_m} = ({\bf x}_1)_{i_1} ({\bf x}_2)_{i_2} \dots ({\bf x}_m)_{i_m},\ i_k = 1,2,\dots,n_k,\ k = 1,2,\dots,m.
$$
When ${\bf x} := {\bf x}_1 = {\bf x}_2 = \dots = {\bf x}_m$, we often abbreviate their outer product as ${\bf x}^{\otimes m}$. Then the (strongly) completely positive tensors are defined as follows.

\begin{definition}[Completely positive \cite{Xu14,LQ2015}]
  Let $\ten{A}$ be an $m^{\rm th}$-order $n$-dimensional symmetric tensor. Then $\ten{A}$ is called completely positive (CP) if it can be written into
  $$
  \ten{A} = \sum_{k=1}^r {\bf u}_k^{\otimes m} \quad
  \text{with } {\bf u}_k \in \mathbb{R}_+^n \text{ for all } k = 1,2,\dots,r;
  $$
  And $\ten{A}$ is called strongly completely positive (SCP) if ${\bf u}_1,{\bf u}_2,\dots,{\bf u}_r$ span the whole space $\mathbb{R}^n$.
\end{definition}

In the even-order case, it is obvious that a CP tensor must be PSD and an SCP tensor must be PD. Not surprisingly, this also holds that odd-order CP (SCP) tensors are symmetric P$_0$- (P-)tensors.

\begin{proposition}\label{prop_cp}
  A completely positive tensor is a P$_0$-tensor, and a strongly completely positive tensor is a P-tensor.
\end{proposition}
\begin{proof}
  We only need to prove the case when $m$ is odd. Let $\ten{A} = \sum_{k=1}^r {\bf u}_k^{\otimes m}$, where ${\bf u}_k \in \mathbb{R}_+^n$ for all $k = 1,2,\dots,r$. Hence, $\ten{A}$ is a CP tensor. For an arbitrary ${\bf x} \in \mathbb{R}^n$, we have
  $$
  \ten{A}{\bf x}^{m-1} = \sum_{k=1}^r {\bf u}_k ({\bf x}^\top {\bf u}_k)^{m-1}.
  $$
  Because $(m-1)$ is even and ${\bf u}_k$'s are nonnegative, the tensor-vector multiplication $\ten{A}{\bf x}^{m-1}$ must also be nonnegative. Thus $\ten{A}$ is a P$_0$-tensor.

  When ${\bf u}_1,{\bf u}_2,\dots,{\bf u}_r$ further span the whole space, for an arbitrary nonzero ${\bf x} \in \mathbb{R}^n$, there must be some ${\bf u}_k$ which is not orthogonal with ${\bf x}$, i.e., ${\bf x}^\top {\bf u}_k \neq 0$. Then there is some index $i$ such that $x_i \neq 0$ and $({\bf u}_k)_i > 0$. Note that $({\bf u}_k)_i > 0$ implies $(\ten{A}{\bf x}^{m-1})_i > 0$. Thereby, $\ten{A}$ is a P-tensor.
\end{proof}

Very recently, Luo and Qi \cite{LQ2015} generalized the doubly nonnegative matrices to higher-order doubly nonnegative tensors with all entries and all $H$-eigenvalues nonnegative. Doubly nonnegative matrices play an important role in the relaxations of some NP-hard optimization problems to nonnegative semidefinite programming. Luo and Qi showed a potential similar effect of doubly nonnegative tensors.

\begin{definition}[Doubly nonnegative]
  A symmetric tensor is called doubly nonnegative (DNN) if both its entries and H-eigenvalues are all nonnegative.
\end{definition}

The relationship between DNN tensors and P$_0$-tensors follows directly by Theorem \ref{thm_heig} and Proposition \ref{prop_psd}.

\begin{proposition}
  A symmetric nonnegative P$_0$-tensor is DNN, and an even-order DNN tensor is a P$_0$-tensor.
\end{proposition}

It is worth pointing out that in the above proposition, the assertion in the second part fails for odd-order tensors, as the following counterexample illustrates.

\begin{example}
  Denote $\ten{A}$ as the symmetric nonnegative $3^{\rm rd}$-order $3$-dimensional tensor with
  \[
  \begin{split}
  &a_{111}=100,\ a_{222}=3,\ a_{333}=1, \\
  &a_{112}=a_{113}=a_{122}=a_{133}=1,\ a_{223}=3,\ a_{233}=2.5, \\
  &a_{123}=0.
  \end{split}
  \]
  Note that for any nonzero ${\bf x}\in \mathbb{R}^3$, $(\ten{A}{\bf x}^2)_1>0$ since the matrix $\ten{A}(1,:,:)$ is positive definite. This implies that all $H$-eigenvalues of $\ten{A}$ is positive. Thus $\ten{A}$ is a DNN tensor. However, by setting ${\bf y}=(0,1,-1)^\top$, it follows that $y_2^2(\ten{A}{\bf y}^2)_2=-0.5$ and $y_3^2(\ten{A}{\bf y}^2)_3=-1$. Thus, $\ten{A}$ is not a P$_0$-tensor by definition.
\end{example}

To conclude, the three existing extensions of positive definiteness to odd orders are more or less involving some extra restrictions: The definition of copositive tensors constrains the vectors in the nonnegative orthant, and CP tensors and DNN tensors must be nonnegative tensor themselves. By contrast, our extension, i.e., symmetric P-tensors, admits no redundant conditions and best corresponds the performance of even-order PSD tensors.

%===============================================================================================================================
%===============================================================================================================================

\section{Structured P- and P$_0$-Tensors}

In this section, we will discuss the relationship between P- or P$_0$-tensors and several important structured tensors. Our results show that the sets of P- and P$_0$-tensors contain a large amount of structured tensors, which is another motivation for us to investigate P- and P$_0$-tensors.

The first kind of structured tensors is H-tensor, whose definition is based on M-tensor. Thus we introduce the concept of M-tensors first. M-tensor is a generalization of famous M-matrix and was widely studied by Zhang et al. \cite{Zhang14} and Ding et al. \cite{Ding13}. Its original definition is as follows, where $\rho(\cdot)$ represents the spectral radius of a tensor, i.e., its largest eigenvalue in absolute value sense.

\begin{definition}[M-tensor]
  Let $\ten{A} = s\ten{I} - \ten{B}$ be an $m^{\rm th}$-order $n$-dimensional tensor, where $\ten{B}$ is nonnegative. If $s \geq \rho(\ten{B})$, then $\ten{A}$ is called an M-tensor. If $s > \rho(\ten{B})$, then $\ten{A}$ is called a nonsingular M-tensor.
\end{definition}

Given an $m^{\rm th}$-order $n$-dimensional tensor $\ten{A}$, the comparison tensor $\ten{M}(\ten{A})$ of $\ten{A}$ is defined by
$$
\big(\ten{M}(\ten{A})\big)_{i_1 i_2 \dots i_m} =
\left\{
\begin{array}{ll}
+|a_{i_1 i_2 \dots i_m}|, & \text{if } i_1 = i_2 = \dots = i_m, \\
-|a_{i_1 i_2 \dots i_m}|, & \text{otherwise}.
\end{array}
\right.
$$
The comparison tensor focuses on the absolute values of the tensor entries rather than their signs. H-tensor was first introduced by Ding et al. \cite{Ding13}, and also studied by Kannan et al. \cite{Berman15} and Li et al. \cite{Li2014}. %\textcolor[rgb]{0,.5,.5}{(I will insert other references on H-tensors soon)}.
We now introduce the definition of H-tensors.

\begin{definition}[H-tensor]
  Let $\ten{A}$ be an $m^{\rm th}$-order $n$-dimensional tensor. If $\ten{M}(\ten{A})$ is an M-tensor, then $\ten{A}$ is called an H-tensor. If $\ten{M}(\ten{A})$ is a nonsingular M-tensor, then $\ten{A}$ is called a nonsingular H-tensor.
\end{definition}

As we have proven, a P- or P$_0$-tensor owns all positive or nonnegative diagonal entries, respectively. Thus we focus on the H-tensors with all nonnegative diagonal entries. The following proposition explains the relationship of H-tensors and P-tensors.

\begin{proposition}\label{thm_hten}
  A nonsingular H-tensor with all positive diagonal entries is a P-tensor, and an H-tensor with all nonnegative diagonal entries is a P$_0$-tensor.
\end{proposition}
\begin{proof}
  Let $\ten{A}$ be a nonsingular H-tensor with all positive diagonal entries. Then its comparison tensor $\ten{M}(\ten{A})$ is a nonsingular M-tensor. Assume that $\ten{A}$ is not a P-tensor, that is, there exists a nonzero vector ${\bf x}$ such that $x_i^{m-1} (\ten{A}{\bf x}^{m-1})_i \leq 0$ for all $i=1,2,\dots,n$. Hence, there is a diagonal matrix $D$ with all nonnegative diagonal entries such that
  $$
  \ten{A}{\bf x}^{m-1} = -D \cdot {\bf x}^{[m-1]}.
  $$
  Denote $\ten{D}$ as the diagonal tensor sharing the same diagonal entries with $D$. Then the above equation is rewritten into $(\ten{A}+\ten{D}){\bf x}^{m-1} = {\bf 0}$. That is, the tensor $\ten{A} + \ten{D}$ has a zero H-eigenvalue.

  Since all the diagonal entries of $\ten{A}$ are positive, the comparison tensor of $\ten{A}+\ten{D}$ is exactly $\ten{M}(\ten{A}) + \ten{D}$. Furthermore, because $\ten{M}(\ten{A})$ is a nonsingular M-tensor, these is a diagonal matrix $S$ with all positive diagonal entries such that $\ten{M}(\ten{A})S^{m-1} = \ten{M}(\ten{A}S^{m-1})$ is strictly diagonally dominant. Thus the tensor $\ten{M}(\ten{A})S^{m-1} + \ten{D}S^{m-1} = \ten{M}\big((\ten{A} + \ten{D})S^{m-1}\big)$ is also strictly diagonally dominant, which implies that $(\ten{A} + \ten{D})S^{m-1}$ is strictly diagonally dominant. This contradicts that $\ten{A} + \ten{D}$ has a zero H-eigenvalue. Thereby, $\ten{A}$ must be a P-tensor.
\end{proof}

A tensor $\ten{A}$ is called diagonally dominant if
$$
|a_{i i \dots i}| \geq \sum_{(i_2,\dots,i_m)\neq(i,\dots,i)} |a_{i i_2 \dots i_m}| \quad \text{for all }i=1,2,\dots,n,
$$
and called strictly diagonally dominant if
$$
|a_{i i \dots i}| > \sum_{(i_2,\dots,i_m)\neq(i,\dots,i)} |a_{i i_2 \dots i_m}| \quad \text{for all }i=1,2,\dots,n.
$$
Denote ${\bf 1} = [1,1,\dots,1]^\top$ as the all-one vector of size $n$. Then the definition of (strictly) diagonally dominant can be rewritten into $\ten{M}(\ten{A}){\bf 1}^{m-1} \geq (>) {\bf 0}$. By Theorems 3 and 16 in \cite{Ding13}, we know that $\ten{M}(\ten{A})$ is thus a (nonsingular) M-tensor, i.e., $\ten{A}$ is a (nonsingular) H-tensor. Therefore, we have the following corollary of Theorem \ref{thm_hten}.

\begin{corollary}\label{coro_dd}
  A strictly diagonally dominant tensor with all positive diagonal entries is a P-tensor, and a diagonally dominant tensor with all nonnegative diagonal entries is a P$_0$-tensor.
\end{corollary}

P$_0$-tensors also have some interesting examples from the spectral theory of uniform hypergraphs (see, e.g., \cite{Hu15,Qi2014,Shao15}). Actually, both the Laplacian tensor and the signless Laplacian tensor of a uniform hypergraph are P$_0$-tensors.

\begin{definition}[Laplacian tensor \cite{Qi2014}]\label{def_sgnlap}
  Let $G =(V, E)$ be an $m$-uniform hypergraph. The \emph{adjacency tensor} of $G$ is defined as the $m^{\rm th}$-order $n$-dimensional tensor $\ten{A}$ whose $(i_1, \ldots, i_m)$th entry is
  $$
  a_{i_1\ldots i_m}=\left\{
  \begin{array}{ll}
    \frac{1}{(m-1)!}, & \hbox{if~$\{i_1,\ldots, i_m\}\in E$;} \\
    0, & \hbox{otherwise.}
  \end{array}
  \right.$$
  Let $\mathcal{D}$ be an $m^{\rm th}$-order $n$-dimensional diagonal tensor with its diagonal element $d_{i\ldots i}$ being $d_i$, the degree of vertex $i$, for all $i\in [n]$. Then ${\mathcal{L}}:={\mathcal{D}}-\ten{A}$ and ${\mathcal{Q}}:={\mathcal{D}}+\ten{A}$ are called the \emph{Laplacian tensor} and \emph{signless Laplacian tensor} of the hypergraph $G$, respectively.
\end{definition}

For even-order hypergraphs, we already know that their (signless) Laplacian tensors are positive semidefinite. However, we do not have this parallel results for odd-order case. Now, we can character this property of both even and odd-order hypergraphs via the concept of P$_0$-tensors. Definition \ref{def_sgnlap} and Corollary \ref{coro_dd} indicate the following corollary.

\begin{corollary}
  The (signless) Laplacian tensor of a uniform hypergraph is a symmetric P$_0$-tensor.
\end{corollary}

Cauchy tensor is another interesting structured tensor (see \cite{Haibin15}). Given a vector ${\bf u}\in \mathbb{R}^n$, we can construct an $m^{\rm th}$-order symmetric Cauchy tensor $\ten{C}$ as
$$
c_{i_1 i_2 \dots i_m} = \frac{1}{u_{i_1} + u_{i_2} + \dots + u_{i_m}},\quad \text{for all }i_k = 1,2,\dots,n,\ k = 1,2,\dots,m,
$$
where ${\bf u}$ is called the generating vector of the Cauchy tensor $\ten{C}$. Luo and Qi \cite{LQ2015} proved that a Cauchy tensor with a positive generating vector is CP and a Cauchy tensor with a mutually distinct positive generating vector is SCP. Then the following corollary of Proposition \ref{prop_cp} is apparent.

\begin{corollary}
  A Cauchy tensor with a positive generating vector is a P$_0$-tensor. Furthermore, when the entries of its generating vector are mutually distinct, it is a P-tensor.
\end{corollary}

Recently, Song and Qi \cite{Song15a} introduced B- and B$_0$-tensors. In matrix case, we know that the set of B-matrices is a subset of the set of P-matrices. Nevertheless, Yuan and You \cite{Yuan14} gave a counterexample to show that a B-tensor is not necessarily a P-tensor for the case that the order is even and larger than $2$ under the Song-Qi's definition. Interestingly, we can prove that an odd-order B-tensor is still a P-tensor under our proposed definition. We first recall the definition of B- and B$_0$-tensors.

\begin{definition}[B-tensor \cite{QS2014}] \label{def_bten}
  An $m^{\rm th}$-order $n$-dimensional tensor $\ten{A}=\left(a_{i_1\ldots i_m}\right)$ is called a B$_0$-tensor if
  $$
  \sum\limits_{i_2,\ldots, i_m\in [n]}a_{ii_2\ldots i_m} \geq 0$$ and $$\frac{1}{n^{m-1}}\sum\limits_{i_2,\ldots,i_m\in [n]}a_{ii_2\ldots i_m}\geq a_{ij_2\ldots j_m},~\forall (j_2,\ldots,j_m)\neq (i,\ldots,i).
  $$
  If strict inequalities hold in the above inequalities, then $\ten{A}$ is called a B-tensor.
\end{definition}

We have the following theorem to construct other P-tensors from a given one in the odd-order case, which also implies the relationship between odd-order B-tensors and P-tensors.

\begin{theorem}
  Let $\ten{A}$ be an $m^{\rm th}$-order $n$-dimensional tensor, and $m$ is odd. If $\ten{A}$ is a P- or P$_0$-tensor and $\ten{C}$ is a completely positive tensor, then $\ten{A} + \ten{C}$ is also a P- or P$_0$-tensor, respectively. Particularly, the addition of a P- or P$_0$-tensor and a nonnegative diagonal tensor is also a P- or P$_0$-tensor, respectively.
\end{theorem}
\begin{proof}
  Denote $\ten{C} = \sum_{k=1}^r ({\bf u}_k)^m$, where ${\bf u}_k$'s are nonnegative vectors and
  $$
  ({\bf u}_k)^m = \underbrace{{\bf u}_k \otimes {\bf u}_k \otimes \dots \otimes {\bf u}_k}_m.
  $$
  Thus $(\ten{A}+\ten{C}){\bf x}^{m-1} = \ten{A}{\bf x}^{m-1} + \sum_{k=1}^r {\bf u}_k \left({\bf u}_k^\top {\bf x}\right)^{m-1}$. Since $m$ is an odd number, the second part $\sum_{k=1}^r {\bf u}_k\left({\bf u}_k^\top {\bf x}\right)^{m-1}$ must be a nonnegative vector. From the definition of P-tensors, if $\ten{A}$ is an odd-order P-tensor, then there must be some index $i$ such that $x_i \neq 0$ and $(\ten{A}{\bf x}^{m-1})_i > 0$. By the above discussion, we know that $((\ten{A}+\ten{C}){\bf x}^{m-1})_i$ is also positive, which implies that $\ten{A} + \ten{C}$ is also a P-tensor.

  Every diagonal tensor $\ten{D}$ with diagonal entries $d_1,d_2,\dots,d_n$ can be decomposed into
  $$
  \ten{D} = \sum_{k=1}^n d_k ({\bf e}_k)^m,
  $$
  where ${\bf e}_k$ is the $k$-th orthonormal column vector. So it is a complete positive tensor, and the results follow directly.
\end{proof}

\begin{corollary}
  An odd-order B-(B$_0$-)tensor is a P-(P$_0$-)tensor, and a symmetric even-order B-(B$_0$-)tensor is a P-(P$_0$-)tensor.
\end{corollary}
\begin{proof} For the first part, let $\ten{A}$ be an odd-order B-tensor. Song and Qi \cite{Song15a} and Yuan and You \cite{Yuan14} proved that a B-tensor can always be decomposed into
  $$
  \ten{A} = \ten{M} + \ten{C},
  $$
  where $\ten{M}$ is a strictly diagonally dominant M-tensor and $\ten{C}$ is the summation of several partially all-one tensors. Since a nonsingular M-tensor is a P-tensor and a partially all-one tensor is completely positive, thus a B-tensor $\ten{A}$ is also a P-tensor. The second part follows readily from \cite{QS2014} and Proposition \ref{prop_psd}.
\end{proof}

%===============================================================================================================================
%===============================================================================================================================

\section{P- and P$_0$-Tensor Cones}

%For simplicity, we use $P(m,n)$, $P_0(m,n)$, $CP_{m,n}$, $SCP_{m,n}$, $NSPSD_{m,n}$, $NSPD_{m,n}$, $PSD_{m,n}$, $PD_{m,n}$, $Z_{m,n}$, $M_{m,n}$, $SM_{m,n}$ to denote the sets of all $m^{\rm th}$-order $n$-dimensional P-tensor, P$_0$-tensors, completely positive tensors, strongly completely positive tensors, positive semidefinite tensors, positive definite tensors, symmetric positive semidefinite tensors, symmetric positive definite tensors, Z-tensors, M-tensors, and nonsingular M-tensors, respectively. The notation $T_{m,n}$ stands for the space of all $m^{\rm th}$-order $n$-dimensional tensors and $S_{m,n}$ for the space of all $m^{\rm th}$-order $n$-dimensional symmetric tensors.

For simplicity, we use the following notations to denote different tensor cones. The orders and the dimensions of the involving tensors are $m$ and $n$, respectively, by default.
\begin{center}
\begin{tabular}{r|l}
  \hline
  \hline
  Abbreviations & Full explanations of the cones \\
  \hline
  $T_{m,n}$ & square tensors \\
  $S_{m,n}$ & symmetric tensors \\
  $P(m,n)$ & P-tensors \\
  $P_0(m,n)$ & P$_0$-tensors \\
  $CP_{m,n}$ & completely positive tensors \\
  $SCP_{m,n}$ & strongly completely positive tensors \\
  $NSPSD_{m,n}$ & positive semidefinite tensors \\
  $NSPD_{m,n}$ & positive definite tensors \\
  $PSD_{m,n}$ & symmetric positive semidefinite tensors \\
  $PD_{m,n}$ & symmetric positive definite tensors \\
  $Z_{m,n}$ & Z-tensors \\
  $M_{m,n}$ & M-tensors \\
  $SM_{m,n}$ & nonsingular M-tensors \\
  \hline
  \hline
\end{tabular}
\end{center}

The following theorem explains that $P_0(m,n)$ is the closure of $P(m,n)$, which is the same relationship with $CP_{m,n}$ and $SCP_{m,n}$, $PSD_{m,n}$ and $PD_{m,n}$, $M_{m,n}$ and $SM_{m,n}$, etc. Therefore, when we come to some results about $P(m,n)$ and those results can be preserved by taking limits, thus those results can also holds for $P_0(m,n)$.

\begin{theorem}
  ${\rm int}\big(P_0(m,n)\big) = P(m,n)$, and thus ${\rm cl}\big(P(m,n)\big) = P_0(m,n)$, where ${\rm int}(\cdot)$ and ${\rm cl}(\cdot)$ denote the interior and the closure of the corresponding set, respectively.
\end{theorem}
\begin{proof}
  On one hand, if $\ten{A}$ is a P$_0$-tensor, then $\ten{A} + \epsilon\ten{I}$ is a P-tensor for all $\epsilon > 0$. Thus each P$_0$-tensor must be the limit of a sequence of P-tensors. On the other hand, if $\ten{A}$ is the limit of a sequence of P-tensors $\{\ten{A}_k\}_{k=1}^\infty$, then we shall prove that $\ten{A}$ must be a P$_0$-tensor. Since each $\ten{A}_k$ is a P-tensor, for an arbitrary ${\bf x} \in \mathbb{R}^n$ there is an index $i_k$ such that $x_{i_k}^{m-1} (\ten{A}_k{\bf x}^{m-1})_{i_k} > 0$. Note that the series $\{i_k\}_{k=1}^\infty$ must admit a convergent subseries $\{i_{k_j}\}_{j=1}^\infty$. Thus there is an index $i = \lim_{j\to\infty} i_{k_j}$ such that $x_i^{m-1} (\ten{A}{\bf x}^{m-1})_i \geq 0$ with $x_i \neq 0$. That is, $\ten{A}$ is a P$_0$-tensor.
\end{proof}

It is known that $P_0(m,n)$ is not convex. Counterexamples can be found in \cite{WXH2010} for the case of matrices. For general high order tensors, the following proposition can also indicate the nonconvexity of $P_0(m,n)$.

\begin{proposition}\label{prop_basis}
  For any $j\in [m]$ and any $i_j\in [n]$, the tensor ${\bf e}^{(i_1)}\otimes {\bf e}^{(i_2)}\otimes \cdots \otimes {\bf e}^{(i_m)}$ is a P$_0$-tensor; For any $j\in [m]$ and any $i_j\in [n]$, if $\delta_{i_1 i_2 \cdots i_m}=0$, then $-{\bf e}^{(i_1)}\otimes \cdots \otimes {\bf e}^{(i_m)}$ is a P$_0$-tensor.
\end{proposition}
\begin{proof}
  For any given $j = 1,2,\dots,m$ and any $i_j = 1,2,\dots,n$, denote $\ten{A}:={\bf e}^{(i_1)}\otimes \cdots \otimes {\bf e}^{(i_m)}$ as the rank-one $m^{\rm th}$-order $n$-dimensional tensor. For any nonzero ${\bf x}\in \mathbb{R}^n$, it follows readily that for any $i = 1,2,\dots,m$,
  \begin{equation}\label{eq_b1}
  x_i^{m-1} (\ten{A}{\bf x}^{m-1})_i =
  \left\{\begin{array}{ll}
  x_{i_1}^{m-1}x_{i_2}\cdots x_{i_m}, & \hbox{if $i=i_1$;} \\
  0, & \hbox{otherwise.}
  \end{array}\right.
  \end{equation}
  Denote ${\rm supp}({\bf x}):=\{i:\, x_i\neq 0\}$.
  \begin{description}
    \item[\bf Case 1] If $i_1 \neq {\rm supp}({\bf x})$, that is, there exists some $k\neq i_1$ such that $x_k\neq 0$. By the equation \eqref{eq_b1}, we have $x_k^{m-1} \ten{A} ({\bf x}^{m-1})_k=0$.
    \item[\bf Case 2] If $i_1 = {\rm supp}({\bf x})$, by setting $k=i_1$, the equation \eqref{eq_b1} implies that
        $$
        x_k^{m-1} (\ten{A}{\bf x}^{m-1})_k =
        \left\{\begin{array}{ll}
        x_{k}^{2(m-1)}, & \hbox{if $\delta_{k i_2\ldots i_m}=1$;} \\
        0, & \hbox{otherwise.}
        \end{array}\right.
        $$
  \end{description}
  Thus, for both the cases, we can always find some index $k = 1,2,\dots,n$ such that $x_k\neq 0$ and  $x_k^{m-1} (\ten{A}{\bf x}^{m-1})_k\geq 0$. By Definition \ref{def_pten}, $\ten{A}$ is a P$_0$-tensor. Similarly, we can show the remaining part.
\end{proof}

For any set $C$, conv$(C)$ stands for the convex hull of $C$ and cone$(C)$ for the convex cone generated by $C$. As an extension of Proposition 4.1 in \cite{WXH2010}, we have the following property for P$_0$-tensor cone by invoking Proposition \ref{prop_basis}.

\begin{proposition}\label{convex-hull} conv$(P_0(m,n))$=cone$(P_0(m,n))=\{\ten{A}\in T_{m,n}: {\rm diag}(\ten{A})\geq {\bf 0}\}$.
\end{proposition}
\begin{proof} By invoking Proposition \ref{prop_basis}, we have
\begin{eqnarray}
 & &{\rm conv}(P_0(m,n))\nonumber\\
&=& {\rm cone}(P_0(m,n)) \nonumber \\
    &\supseteq & {\rm cone}\left(\{E^{i_1\ldots i_m}, i_j\in [n], j\in[m]\}\cup \{-E^{i_1\ldots i_m}, i_j\in [n], j\in[m], \delta_{i_1\cdots i_m}=0\}\right) \nonumber\\
    &=& \left\{\sum\limits_{i_1,\ldots, i_m\in [n]}a_{i_1\ldots i_m}E^{i_1\ldots i_1}+\sum_{\begin{subarray}{c}
i_1,\ldots, i_m\in [n]\nonumber\\
\delta_{i_1\cdots i_m}=0\nonumber
\end{subarray}} -b_{i_1\ldots i_m}E^{i_1\ldots i_m}: a_{i_1 \ldots i_1}\geq 0, b_{i_1 \ldots i_m}\in {\mathbb{R}}, \forall i_1,\ldots, i_m\in [n]\right\} \nonumber\\
    &=& \left\{\sum\limits_{i\in [n]}a_{i_1\ldots i_1}E^{i_1\ldots i_1}+\sum_{\begin{subarray}{c}
i_1,\ldots, i_m\in [n]\nonumber\\
\delta_{i_1\cdots i_m}=0\nonumber
\end{subarray}} c_{i_1\ldots i_m}E^{i_1\ldots i_m}: a_{i_1 \ldots i_1}\geq 0, c_{i_1 \ldots i_m}\in {\mathbb{R}}, \forall i_1,\ldots, i_m\in [n] \right\}  \nonumber\\
    &=&  \{\ten{A}\in T_{m,n}: {\rm diag}(\ten{A})\geq 0\} \nonumber
\end{eqnarray}
On the other hand, for any P$_0$-tensor $\ten{A}$, we have ${\rm diag}(\ten{A})\geq 0$ from Theorem \ref{thm_subten}. Thus,
$${\rm conv}(P_0(m,n))\subseteq {\rm conv}(\{\ten{A}\in T_{m,n}: {\rm diag}(\ten{A})\geq 0\})=\{\ten{A}\in T_{m,n}: {\rm diag}(\ten{A})\geq 0\}.$$
This completes the proof.\end{proof}

The following two propositions are about the relationship between P-tensors and PSD tensors.

\begin{proposition}[\cite{Song15a}]
  Let $m$ be even. We have $$P(m,n)\cap S_{m,n}=PD_{m,n}~{\rm and}~~P_0(m,n)\cap S_{m,n}=PSD_{m,n}.$$
\end{proposition}

\begin{proposition}
  $NSPSD_{m,n}$ is a convex cone in $T_{m,n}$. Moreover,
  $$
  {\rm int}(NSPSD_{m,n}) = NSPD_{m,n}, \quad
  NSPD_{m,n} \subset P(m,n), \quad
  NSPSD_{m,n} \subset P_0(m,n).
  $$
\end{proposition}
\begin{proof}
  The convexity of $NSPSD_{m,n}$ and the assertion ${\rm int}\big(NSPSD_{m,n}\big) = NSPD_{m,n}$ follow readily from Definition \ref{def_psd}. And the remaining parts are obtained by invoking Proposition \ref{prop_psd}.
\end{proof}

\begin{corollary} $P(m,n)\cap Z_{m,n}=SM_{m,n}$, and $P_0(m,n)\cap Z_{m,n}=M_{m,n}$.
\end{corollary}

\begin{theorem} $P(m,n)\cap CP_{m,n}=SCP_{m,n}$.
\end{theorem}
\begin{proof} To get the inclusion $SCP_{m,n}\subseteq P(m,n)\cap CP_{m,n}$, it suffices to show that any tensor $\ten{A}\in  SCP_{m,n}$ is a P-tensor. By the definition of strongly completely positive tensor, we can find some nonnegative vectors ${\bf u}_k$, $k=1,\ldots, r$, such that $\ten{A}=\sum\limits_{k=1}^r\left({\bf u}_k\right)^m$ and span$\{{\bf u}_1,\ldots, {\bf u}_r\}=\mathbb{R}^n$. Assume on the contrary that $\ten{A}$ is not a P-tensor, that is, there exists some nonzero ${\bf y}\in \R^n$ such that for any $i\in \textrm{supp}({\bf y}):=\{i:y_i\neq 0\}$, \begin{equation}\label{d} y_i^{m-1}(\ten{A}{\bf y}^{m-1})_i\leq 0. \end{equation} When $m$ is even, it follows readily from (\ref{d}) that
$$0\geq \sum\limits_{i\in {\rm supp}({\bf y})} y_i (\ten{A}{\bf y}^{m-1})_i=\ten{A} {\bf y}^m=\sum\limits_{k=1}^r\left({\bf y}^T{\bf u}_k\right)^m\geq 0.$$
Thus, $ \ten{A} {\bf y}^m=0$. Together with span$\{{\bf u}_1,\ldots, {\bf u}_r\}=\mathbb{R}^n$, we have ${\bf y}={\bf 0}$, which is a contradiction. When $m$ is odd, (\ref{d}) indicates that for all $i\in {\rm supp}({\bf y})$,
$$0\geq (\ten{A}{\bf y}^{m-1})_i= \sum\limits_{k=1}^ru^{(k)}_i\left({\bf y}^T{\bf u}_k\right)^{m-1}\geq 0.$$ This indicates $$ u^{(k)}_i({\bf y}^T{\bf u}_k)=0, ~\forall i\in {\rm supp}({\bf y}),~\forall k\in [r].$$  Thus, $$({\bf y}^T{\bf u}_k)^2=\sum\limits_{i\in  {\rm supp}({\bf y}) } y_iu^{(k)}_i\left({\bf y}^T{\bf u}_k\right)=0, \forall k\in [r].$$ Since span$\{{\bf u}_1,\ldots, {\bf u}_r\}=\mathbb{R}^n$, we have ${\bf y}={\bf 0}$, which is a contradiction. Therefore, $\ten{A}$ is a P-tensor. The remaining inclusion follows readily from Theorem \ref{thm_heig} and the definition of strongly completely positive tensors.
\end{proof}

%===============================================================================================================================
%===============================================================================================================================

\section{Tensor Complementarity Problems and Some Other Applications}

The tensor complementarity problem (see \cite{Che15,Luo15,Song15c,Song15b}) is referred to finding some vector ${\bf x}$ which satisfying
\begin{equation}\label{eq_tcp}
{\bf x} \geq {\bf 0},\ \ten{A}{\bf x}^{m-1} + {\bf q} \geq {\bf 0},\ \langle {\bf x}, \ten{A}{\bf x}^{m-1} + {\bf q} \rangle = 0,
\end{equation}
which is a generalization of the linear complementarity problem and a special case of the nonlinear complementarity problem. The P-matrices play an essential role in linear complementarity problems. That is, if $A$ is a P-matrix, then the following linear complementarity problem
$$
{\bf x} \geq {\bf 0},\ A{\bf x} + {\bf q} \geq {\bf 0},\ \langle {\bf x}, A{\bf x} + {\bf q} \rangle = 0,
$$
has a unique solution for an arbitrary vector ${\bf q}$. We will see shortly that if the tensor $\ten{A}$ in the Problem \eqref{eq_tcp} is a P-tensor, then we have some characterization of its solution set, although the solution is never necessarily unique. Before that, we need an important lemma from Mor{\'e} \cite{More74}.

\begin{lemma}[\cite{More74}]
  Let ${\bf f}:\mathbb{R}^n \to \mathbb{R}^n$ be a continuous mapping on the rectangular cone $K_\mathbb{R}$, and assume that for each ${\bf x} \neq {\bf 0}$ in $K_\mathbb{R}$,
  $$
  \max_{i=1,2,\dots,n} x_i \big(f_i({\bf x}) - f_i({\bf 0})\big) > 0.
  $$
  If the mapping ${\bf g}:\mathbb{R}^n \to \mathbb{R}^n$ defined by ${\bf g}({\bf x}) = {\bf f}({\bf x}) - {\bf f}({\bf 0})$ is positively homogeneous, then for each ${\bf z} \in \mathbb{R}^n$, there is an ${\bf x}^\ast \geq_{K_\mathbb{R}} {\bf 0}$ with
  \begin{equation}\label{eq_ncp}
  {\bf f}({\bf x}^\ast) \geq_{K_\mathbb{R}^\ast} {\bf z} \quad\text{and}\quad \langle {\bf x}^\ast, {\bf f}({\bf x}^\ast) - {\bf z} \rangle = 0.
  \end{equation}
  Moreover, the set of ${\bf x}^\ast \geq_{K_\mathbb{R}} {\bf 0}$ which satisfy \eqref{eq_ncp} is compact.
\end{lemma}

Let $K_\mathbb{R}=\mathbb{R}^n_+$ in the above lemma and denote ${\bf f}({\bf x}) = \ten{A}{\bf x}^{m-1} + {\bf q}$, where $\ten{A}$ is a P-tensor. Since ${\bf f}({\bf 0}) = {\bf q}$, then the definition of P-tensors and the homogeneousness of ${\bf f}({\bf x}) - {\bf f}({\bf 0}) = \ten{A}{\bf x}^{m-1}$ implies that ${\bf f}({\bf x})$ satisfies the conditions of the above lemma. Therefore, we have the following theorem about the tensor complementarity problems with P-tensors.

\begin{theorem}\label{thm_ptcp}
  If $\ten{A}$ is a P-tensor, then the complementarity problem \eqref{eq_tcp} has a nonempty compact solution set.
\end{theorem}

We have another class of special tensors which guarantees the feasibility of the tensor complementarity problem \eqref{eq_tcp}.
\begin{definition}[S-tensor]
  Let $\ten{A}$ be an $m^{\rm th}$-order $n$-dimensional tensor. Then there exists a vector ${\bf x} > {\bf 0}$ such that $\ten{A}{\bf x}^{m-1} > {\bf 0}$, then we call $\ten{A}$ an S-tensor.
\end{definition}

From Theorem \ref{thm_ptcp}, we can conclude that P-tensors are special S-tensors.

\begin{theorem}
  A P-tensor is also an S-tensor.
\end{theorem}
\begin{proof}
  Let $\ten{A}$ be a P-tensor, then from Theorem \ref{thm_ptcp} we know that there must exist some ${\bf x} \geq {\bf 0}$ such that
  $$
  \ten{A}{\bf x}^{m-1} \geq {\bf 1} \quad\text{and}\quad \langle {\bf x}, \ten{A}{\bf x}^{m-1} - {\bf 1} \rangle = 0,
  $$
  where ${\bf 1}$ denotes the all-one vector of the corresponding size. Then for an arbitrary $\epsilon >0$, we have ${\bf x} + \epsilon{\bf 1} > 0$ and
  $$
  \ten{A}({\bf x} + \epsilon{\bf 1})^{m-1} = \ten{A}{\bf x}^{m-1} + O(\epsilon) \geq {\bf 1} + O(\epsilon).
  $$
  When $\epsilon$ is small enough, the tensor-vector multiplication $\ten{A}({\bf x} + \epsilon{\bf 1})^{m-1} > {\bf 0}$. This completes the proof.
\end{proof}

%{\bf Question:} \textcolor[rgb]{0,.5,.5}{Let $f({\bf x})$ be a smooth function on $\mathbb{R}^n$. We know that a positive definite Hessian matrix $\nabla^2 f({\bf x})$ implies that $f$ is convex in a neighbourhood of ${\bf x}$. Then what can we say about $f$ if $\nabla^3 f({\bf x})$ is a symmetric third-order P-tensor?}

%{\bf Partial Answer:}
As symmetric real tensors are closely related to higher order derivatives of sufficiently differentiable multivariate real-valued functions, we have the following property of symmetric P-tensors.

\begin{proposition}\label{derivative} For any given real-valued function $f:\R^n\rightarrow \R$ which is third-order continuously differentiable, any given vector ${\bf x}\in\R^n$, denote $\nabla f({\bf x})$ and $\nabla^2 f({\bf x})$ as the gradient vector and the Hessian matrix of $f$ at ${\bf x}$ respectively, and $\nabla^3 f({\bf x})$ as the third-order derivative tensor with
$$\nabla^3 f({\bf x}) {\bf d}^3:=\sum\limits_{i_1,i_2,i_3\in [n]}\frac{\partial^3 f}{\partial x_{i_1} \partial x_{i_2} \partial x_{i_3}} d_{i_1}d_{i_2}d_{i_3}.$$ If $\nabla f({\bf x})={\bf 0}$, $\nabla^2 f({\bf x})=0$, and $\nabla^3 f({\bf x})$ is a P-tensor at some point ${\bf x}$, then for any nonzero $d\in\R^n_+$, we can find some sufficiently small $\alpha({\bf d})>0$ such that $f({\bf x}+\alpha({\bf d}){\bf d})>f({\bf x})>f({\bf x}-\alpha({\bf d}){\bf d})$.
\end{proposition}
\begin{proof} Write the third-order Taylor expansion of $f$ at $x$ as follows:
\begin{equation}\label{Taylor}
f({\bf x}+\alpha({\bf d}){\bf d})=f({\bf x})+\alpha({\bf d})\nabla f({\bf x})^T {\bf d}+\frac{\alpha({\bf d})^2}{2}{\bf d}^T\nabla^2 f({\bf x}) {\bf d}+\frac{\alpha({\bf d})^3}{6}\nabla^3 f({\bf x}) {\bf d}^3+o\left(\alpha({\bf d})^3\|{\bf d}\|^3\right),
\end{equation} By employing Corollary \ref{prop_cop}, $\nabla^3 f({\bf x})$ is a strictly copositive tensor. This immediately implies that for any nonzero $d\in\R^n_+$, $\nabla^3 f({\bf x}) {\bf d}^3>0$. Combining with the conditions $\nabla f({\bf x})={\bf 0}$ and $\nabla^2 f({\bf x})=0$, the Taylor expansion (\ref{Taylor}) implies the desired assertion.\end{proof}

\begin{remark} By using the similar scheme of proof, the result in Proposition \ref{derivative} can be generalized to any odd-order case larger than $3$.
\end{remark}

%$$
%\phi_{\bf d}(\lambda) =
%f({\bf x} + \lambda {\bf d}) =
%f({\bf x}) +
%\lambda {\bf d}^\top \nabla f({\bf x}) +
%\textstyle\frac{\lambda^2}{2} {\bf d}^\top \nabla^2 f({\bf x}) {\bf d} +
%\textstyle\frac{\lambda^3}{3!} \big[\nabla^3 f({\bf x})\big] {\bf d}^3 +
%O(\epsilon^4)
%$$

%$$
%\phi_{\bf d}^{(k)}(0) = \big[\nabla^k f({\bf x})\big] {\bf d}^k,
%$$

%===============================================================================================================================
%===============================================================================================================================

\section{Conclusions}
By employing a homogenous formula, the concept of higher-order P-tensor has been introduced which is well-defined for all even and odd order tensors. This class of tensors has been shown to contain many important subclasses of tensors, including the well-known positive definite tensors, the nonsingular M-tensors, the nonsingular H-tensors with positive diagonal entries, the strictly diagonally dominant tensors with positive diagonal entries, and so on. As an extension of PSD tensors, the P$_0$-tensors and their relationship among other PSD extensions have been characterized. Additionally, many structured P- (P$_0$-) tensors have been stated and many related tensor cones have been analyzed. All of these can be treated as refinement components for tensor theory and analysis. As an important application of P-tensors, the tensor complementarity problems with P-tensors have been shown to admit nonempty and compact solution sets. The application in multivariate real-valued functions has also been discussed.

%===============================================================================================================================
%===============================================================================================================================

\bibliographystyle{Abbrv}
\bibliography{bib_PTensor}

\begin{thebibliography}{10}

\bibitem{Anderson75}
B.~Anderson, N.~K. Bose, and E.~I. Jury.
\newblock Output feedback stabilization and related problems-solution via
  decision methods.
\newblock {\em Automatic Control, IEEE Transactions on}, 20(1):53--66, 1975.

\bibitem{Berman94}
A.~Berman and R.~J. Plemmons.
\newblock {\em Nonnegative matrices in the mathematical sciences}, volume~9 of
  {\em Classics in Applied Mathematics}.
\newblock Society for Industrial and Applied Mathematics (SIAM), Philadelphia,
  PA, 1994.
\newblock Revised reprint of the 1979 original.

\bibitem{Che15}
M.~Che, L.~Qi, and Y.~Wei.
\newblock Positive definite tensors to nonlinear complementarity problems.
\newblock {\em J. Optimiz. Theory App.}, 2015.
\newblock doi: 10.1007/s10957-015-0773-1.

\bibitem{Haibin15}
H.~Chen and L.~Qi.
\newblock Positive definiteness and semi-definiteness of even order symmetric
  {C}auchy tensors.
\newblock {\em J. Ind. Manag. Optim.}, 11(4):1263--1274, 2015.

\bibitem{CQi2014}
Z.~Chen and L.~Qi.
\newblock Circulant tensors with applications to spectral hypergraph theory and
  stochastic process.
\newblock {\em arXiv preprint arXiv:1312.2752}, 2014.

\bibitem{Cottle92}
R.~W. Cottle, J.-S. Pang, and R.~E. Stone.
\newblock {\em The linear complementarity problem}.
\newblock Computer Science and Scientific Computing. Academic Press, Inc.,
  Boston, MA, 1992.

\bibitem{Ding13}
W.~Ding, L.~Qi, and Y.~Wei.
\newblock {$\mathcal{M}$}-tensors and nonsingular {$\mathcal{M}$}-tensors.
\newblock {\em Linear Algebra Appl.}, 439(10):3264--3278, 2013.

\bibitem{Pang03a}
F.~Facchinei and J.-S. Pang.
\newblock {\em Finite-dimensional variational inequalities and complementarity
  problems. {V}ol. {I}}.
\newblock Springer Series in Operations Research. Springer-Verlag, New York,
  2003.

\bibitem{Pang03b}
F.~Facchinei and J.-S. Pang.
\newblock {\em Finite-dimensional variational inequalities and complementarity
  problems. {V}ol. {II}}.
\newblock Springer Series in Operations Research. Springer-Verlag, New York,
  2003.

\bibitem{Fiedler62}
M.~Fiedler and V.~Pt{\'a}k.
\newblock On matrices with non-positive off-diagonal elements and positive
  principal minors.
\newblock {\em Czechoslovak Math. J.}, 12 (87):382--400, 1962.

\bibitem{Hu15}
S.~Hu, L.~Qi, and J.~Xie.
\newblock The largest {L}aplacian and signless {L}aplacian {H}-eigenvalues of a
  uniform hypergraph.
\newblock {\em Linear Algebra Appl.}, 469:1--27, 2015.

\bibitem{Li2014}
C.~Li, F.~Wang, J.~Zhao, Y.~Zhu, and Y.~Li.
\newblock Criterions for the positive definiteness of real supersymmetric
  tensors.
\newblock {\em J. Comput. Appl. Math.}, 255:1--14, 2014.

\bibitem{HuangTZ07}
H.-B. Li, T.-Z. Huang, and H.~Li.
\newblock On some subclasses of {$P$}-matrices.
\newblock {\em Numer. Linear Algebra Appl.}, 14(5):391--405, 2007.

\bibitem{LQ2015}
Z.~Luo and L.~Qi.
\newblock Doubly nonnegative tensors, completely positive tensors and
  applications.
\newblock {\em arXiv preprint arXiv:1504.07806}, 2015.

\bibitem{Luo15}
Z.~Luo, L.~Qi, and N.~Xiu.
\newblock The sparsest solutions to {Z}-tensor complementarity problems.
\newblock {\em arXiv preprint arXiv:1505.00993}, 2015.

\bibitem{More74}
J.~J. Mor{\'e}.
\newblock Coercivity conditions in nonlinear complementarity problems.
\newblock {\em SIAM Rev.}, 16:1--16, 1974.

\bibitem{Pena09}
J.~M. Pe{\~n}a.
\newblock Eigenvalue bounds for some classes of {$P$}-matrices.
\newblock {\em Numer. Linear Algebra Appl.}, 16(11-12):871--882, 2009.

\bibitem{Pena11}
J.~M. Pe{\~n}a.
\newblock Diagonal dominance, {S}chur complements and some classes of
  {$H$}-matrices and {$P$}-matrices.
\newblock {\em Adv. Comput. Math.}, 35(2-4):357--373, 2011.

\bibitem{Qi05}
L.~Qi.
\newblock Eigenvalues of a real supersymmetric tensor.
\newblock {\em J. Symbolic Comput.}, 40(6):1302--1324, 2005.

\bibitem{Qi13}
L.~Qi.
\newblock Symmetric nonnegative tensors and copositive tensors.
\newblock {\em Linear Algebra Appl.}, 439(1):228--238, 2013.

\bibitem{Qi2014}
L.~Qi.
\newblock {$H^+$}-eigenvalues of {L}aplacian and signless {L}aplacian tensors.
\newblock {\em Commun. Math. Sci.}, 12(6):1045--1064, 2014.

\bibitem{QS2014}
L.~Qi and Y.~Song.
\newblock An even order symmetric {B} tensor is positive definite.
\newblock {\em Linear Algebra Appl.}, 457:303--312, 2014.

\bibitem{Qi00}
L.~Qi, D.~Sun, and G.~Zhou.
\newblock A new look at smoothing {N}ewton methods for nonlinear
  complementarity problems and box constrained variational inequalities.
\newblock {\em Math. Program.}, 87(1, Ser. A):1--35, 2000.

\bibitem{Xu14}
L.~Qi, C.~Xu, and Y.~Xu.
\newblock Nonnegative tensor factorization, completely positive tensors, and a
  hierarchical elimination algorithm.
\newblock {\em SIAM J. Matrix Anal. Appl.}, 35(4):1227--1241, 2014.

\bibitem{QYW10}
L.~Qi, G.~Yu, and E.~X. Wu.
\newblock Higher order positive semi-definite diffusion tensor imaging.
\newblock {\em SIAM J. Imaging Sci.}, 3(3):416--433, 2010.

\bibitem{Berman15}
M.~Rajesh~Kannan, N.~Shaked-Monderer, and A.~Berman.
\newblock Some properties of strong {$\mathcal{H}$}-tensors and general
  {$\mathcal{H}$}-tensors.
\newblock {\em Linear Algebra Appl.}, 476:42--55, 2015.

\bibitem{Rump03}
S.~M. Rump.
\newblock On {$P$}-matrices.
\newblock {\em Linear Algebra Appl.}, 363:237--250, 2003.
\newblock Special issue on nonnegative matrices, $M$-matrices and their
  generalizations (Oberwolfach, 2000).

\bibitem{Shao15}
J.-Y. Shao, L.~Qi, and S.~Hu.
\newblock Some new trace formulas of tensors with applications in spectral
  hypergraph theory.
\newblock {\em Linear Multilinear Algebra}, 63(5):971--992, 2015.

\bibitem{Song99}
Y.~Song, M.~S. Gowda, and G.~Ravindran.
\newblock On some properties of {{\bf P}}-matrix sets.
\newblock {\em Linear Algebra Appl.}, 290(1-3):237--246, 1999.

\bibitem{Song15d}
Y.~Song and L.~Qi.
\newblock Necessary and sufficient conditions for copositive tensors.
\newblock {\em Linear Multilinear Algebra}, 63(1):120--131, 2015.

\bibitem{Song15a}
Y.~Song and L.~Qi.
\newblock Properties of some classes of structured tensors.
\newblock {\em J. Optim. Theory Appl.}, 165(3):854--873, 2015.

\bibitem{Song15c}
Y.~Song and L.~Qi.
\newblock Properties of tensor complementarity problem and some classes of
  structured tensors.
\newblock {\em arXiv preprint arXiv:1412.0113}, 2015.

\bibitem{Song15b}
Y.~Song and L.~Qi.
\newblock Tensor complementarity problem and semi-positive tensors.
\newblock {\em arXiv preprint arXiv:1502.02209}, 2015.

\bibitem{WXH2010}
Y.~Wang, N.~Xiu, and J.~Han.
\newblock On cone of nonsymmetric positive semidefinite matrices.
\newblock {\em Linear Algebra Appl.}, 433(4):718--736, 2010.

\bibitem{Yuan14}
P.~Yuan and L.~You.
\newblock Some remarks on {$P$}, {$P_0$}, {$B$} and {$B_0$} tensors.
\newblock {\em Linear Algebra Appl.}, 459:511--521, 2014.

\bibitem{Zhang14}
L.~Zhang, L.~Qi, and G.~Zhou.
\newblock {$M$}-tensors and some applications.
\newblock {\em SIAM J. Matrix Anal. Appl.}, 35(2):437--452, 2014.

\end{thebibliography}

\end{document}